\documentclass[smallextended]{svjour3}   
\smartqed  
\usepackage{graphicx}
\usepackage{amsmath,amssymb,latexsym,enumerate,mathrsfs}
\usepackage{hyperref}

\DeclareMathOperator*{\argmax}{arg\,max}

\newtheorem{assumption}[theorem]{\bf Assumption}

\def\etal{\mbox{et al.}}
\usepackage{graphicx}
\journalname{}

\begin{document}

\title{Optimal control of diffusion processes pertaining to an opioid epidemic dynamical model with random perturbations}

\titlerunning{Optimal control of diffusion processes pertaining to an opioid epidemic dynamical model} 

\author{Getachew K. Befekadu \and
Quanyan Zhu
}

\institute{Getachew K. Befekadu \at
              Department of Mechanical and Aerospace Engineering\\
              University of Florida - REEF, 1350 N. Poquito Rd, Shalimar, FL 32579, USA. \\
              Tel.: +1 850 833 9350\\
              Fax: +1 850 833 9366\\
              \email{gbefekadu@ufl.edu}    
               \and
           Quanyan Zhu \at
           Department of Electrical and Computer Engineering\\
           Tandon School of Engineering\\
           New York University\\
           5 MetroTech Center, LC 200A, Brooklyn, NY 11201, USA.\\
            \email{quanyan.zhu@nyu.edu}
}

\date{Received: date / Accepted: date}

\maketitle

\begin{abstract}
In this paper, we consider the problem of controlling a diffusion process pertaining to an opioid epidemic dynamical model with random perturbation so as to prevent it from leaving a given bounded open domain. Here, we assume that the random perturbation enters only through the dynamics of the susceptible group in the compartmental model of the opioid epidemic dynamics and, as a result of this, the corresponding diffusion is degenerate, for which we further assume that the associated diffusion operator is hypoelliptic. In particular, we minimize the asymptotic exit rate of such a controlled-diffusion process from the given bounded open domain and we derive the Hamilton-Jacobi-Bellman equation for the corresponding optimal control problem, which is closely related to a nonlinear eigenvalue problem. Finally, we also prove a verification theorem that provides a sufficient condition for optimal control.
\end{abstract}

 

\section{Introduction} \label{S1}

The opioid drug-related problem has recently reached crisis levels worthy of declaring a public health emergency in the United States. For example, more than 53,000 people in the United States died from an opioid overdose in 2016 -- more than double the figure in 2010 -- and the increasing use, misuse and abuse of heroin, fentanyl and other opiates, including prescription drugs, shows no signs of slowing (e.g., see \cite{RudSDS16} and \cite{OASPE15} for additional discussions). In response to this, a number of federal and state agencies throughout the United States have implemented a wide range of opioid-related policies\footnote{Including the Ryan Haight Online Pharmacy Consumer Protection Act of 2008 which prohibited the Internet distribution of controlled substances without a valid prescription \cite{r5}; see also \cite{DowHC16} for CDC guideline for prescribing opioids for chronic pain -- United States, 2016.} that are primarily aimed at curbing prescription opioid abuse, establishing guidelines to prevent inappropriate prescribing practices, developing abuse deterrents or preventing drug diversion mechanisms \cite{r1}, \cite{r2} and \cite{r3}. On the other hand, only a few studies have been reported on the need for effective intervention strategies, based on mathematical optimal control theory of epidemiology for infectious diseases, with the intent of better understanding the dynamics of the current serious opioid epidemic (e.g., see \cite{r6}, \cite{r7} and \cite{BefZ18} in context of exploring the dynamics of drug abuse epidemics, focusing on the interplay between the different opioid user groups and the process of rehabilitation and treatment from addiction; see also \cite{r8}, \cite{r9} or \cite{r12} for additional studies, but in the context of heroin epidemics that resembling the classic susceptible-infected-recovered (SIR) model, based on the work of \cite{r13}). Here, we would like to point out that the roots to opioid crisis are complex and tangled with social and political issues; and therefore, only systemic research and evidence-based strategies can identify the most effective ways for intervention of the current opioid crisis.

In this paper, we consider optimal control of an opioid epidemic dynamical model, when there is a random perturbation that enters through the dynamics of the susceptible group in the compartmental model of the opioid epidemic dynamics. Note that the random noise enters only through a particular subsystem in the compartmental model, and then its effect is subsequently propagated to the other subsystems. As a consequence, the corresponding diffusion is degenerate, for which we also assume that the associated diffusion operator is hypoelliptic, where such a hypoellipticity assumption implies a strong accessibility property of controllable nonlinear systems that are driven by white noise (e.g., see \cite{SusJu72} concerning the controllability of nonlinear systems, which is closely related to \cite{StrVa72} and \cite{IchKu74}; see also \cite[Section~3]{Ell73} and \cite{Ama79} for additional discussions from the view point of the control theory). Here, our main objective is to prevent the controlled-diffusion process pertaining to the randomly perturbed opioid epidemic dynamics from leaving a given bounded open domain. To this end, we minimize the asymptotic exit rate (as opposed to maximizing the mean exit-time) with which the controlled-diffusion process exits from the given bounded open domain, and we further derive the Hamilton-Jacobi-Bellman (HJB) equation for the corresponding optimal control problem, which is also closely related to a nonlinear eigenvalue problem. Moreover, we also prove a verification theorem that provides a sufficient condition for the solution of optimal control.

The remainder of this paper is organized as follows. In Section~\ref{S2}, we present the problem formulation for optimal control of a diffusion process pertaining to an opioid epidemic dynamical model with random perturbation. The problem we focus on is to minimize the asymptotic exit rate with which the controlled-diffusion process exits from the given bounded open domain. In Section~\ref{S3}, we provide our main results -- where we derive the HJB equation for the corresponding optimal control problem. In this section, we also provide a verification theorem for the optimal control. Finally, Section~\ref{S4} provides further remarks.
 
\section{Problem formulation} \label{S2}
In this section, we present the problem formulation for optimal control of a diffusion process pertaining to an opioid epidemic dynamical model with random perturbation. In particular, the problem we focus on is to minimize the asymptotic exit rate with which the controlled-diffusion process exits from the given bounded open domain and we further estabilish a connection with a nonlinear eigenvalue problem

\subsection{Mathematical model} \label{S2(1)}
In this subsection, we consider an opioid epidemic dynamical model that describes the interplay between regular prescription opioid use, addictive use, and the process of rehabilitation and relapsing into opioid drug use (e.g., see \cite{r6} for a detailed discussion). To this end, we introduce the following population groups
\begin{enumerate} [(i)]
\item {\it Susceptible group} - $S$: This group in the compartmental model includes those who are susceptible to opioid addiction, but they are not currently using opioids. In the compartmental model, everyone who is not in addiction treatment, already an addict, or using opioids as medically prescribed is classified as ``susceptible".
\item {\it Prescribed user group} - $P$: This group in the compartmental model is composed of individuals who have health related concerns and also have access to opioids through a proper physician's prescription, but they are not addicted to opioids. Members of this group have some inherent tendency of becoming addicted to their prescribed opioids.
\item {\it Addiction user group} - $A$: This group in the compartmental model is composed of people who are addicted to opioids. There are multiple interaction routes to this group in the compartmental model, including those routes that are bypassing the prescribed user group $P$ (see also Fig~\ref{Fig1} that shows the relationships between the different groups).
\item {\it Treatment/rehabilitation} - $R$: This group in the compartmental model contains individuals who are in treatment for their addiction. Here, we include an inherent rate of falling back into addiction as well as a typical process of relapsing due to general availability of the drug. Moreover, we also assume that some of the members from the recovering group who have completed their treatment may return to being susceptible. That is, we assume that successful treatment does not imply permanent immunity to addiction (i.e., in general, an assumption based on the balance of increased risk of addiction verses increased awareness and avoidance).
\end{enumerate}
Then, using the basic remarks made above, we specify the following SIR compartmental model for the opioid epidemic dynamics described by the following four continuous-time differential equations 
\begin{align}
\dot{S}(t) = -\alpha S(t) &- \beta(1-\xi)S(t) A(t) - \beta \xi S(t) P(t) + \epsilon P(t) \notag \\
                        &\quad  + \delta R(t) + \mu(P(t)+R(t)) + \mu^{\ast} A(t), \label{Eq2.1a}
\end{align}
\begin{align}
\dot{P}(t) &= \alpha S(t) - (\epsilon + \gamma + \mu) P(t), \label{Eq2.1b}
\end{align}
\begin{align}
\dot{A}(t) =  \gamma P(t) + \sigma R(t) &+ \beta(1 - \xi) S(t) A(t) + \beta \xi S(t) P(t) \notag \\
                                                    & + \nu R(t) A(t) - (\zeta + \mu^{\ast}) A(t) \label{Eq2.1c}
\end{align}
and
\begin{align}
\dot{R}(t) &= \zeta A(t) - \mu R(t) A(t) - (\delta + \sigma + \mu) R(t), \label{Eq2.1d} 
\end{align}
where the normalized overall population is assumed to be constant, i.e., $1= A(t)+S(t)+R(t)+P(t)$, since the number of mortality due to opioid-related overdose is very small, when compared to the change in the total population numbers in the short term. Moreover, the followings are brief description for the parameters in the above system of equations, i.e., the system parameters in Equations~\eqref{Eq2.1a}--\eqref{Eq2.1d},
{\em
\begin{itemize}
\item $\alpha S(t)$: the rate at which people are prescribed opioids.
\item $\beta$: the total probability of becoming addicted to opioids other than by prescription.
\item $\beta(1-\xi)$: the proportion of $\beta$ caused by black market drugs or other addicts.
\item $\beta \xi$: the rate at which the non-prescribed, susceptible population begins abusing opioids due to the accessibility of extra prescription opioids, e.g., new addicts got the drug from a friend or relative's prescription.
\item $\epsilon$: the rate at which people come back to the susceptible group after being prescribed opioids.
\item $\delta$: the rate at which people come back to the susceptible group after successfully finishing treatment. Despite having completed rehabilitation, we assume people are susceptible to addiction for life.
\item $\mu$: the natural death rate.
\item $\mu^{\ast}$: the (enhanced) death rate for addicts ($\mu$ plus overdose rate).
\item $\gamma$: the rate at which the prescribed opioid users fall into addiction.
\item $\zeta$: the rate at which addicted/dependent opioid users enter the treatment/ rehabilitation process.
\item $\nu$: the rate at which users during the treatment fall back into addictive drug use due to the availability of prescribed painkillers from friends or relatives.
\end{itemize}}

\begin{figure}[ht]
\includegraphics[width=3.5 in]{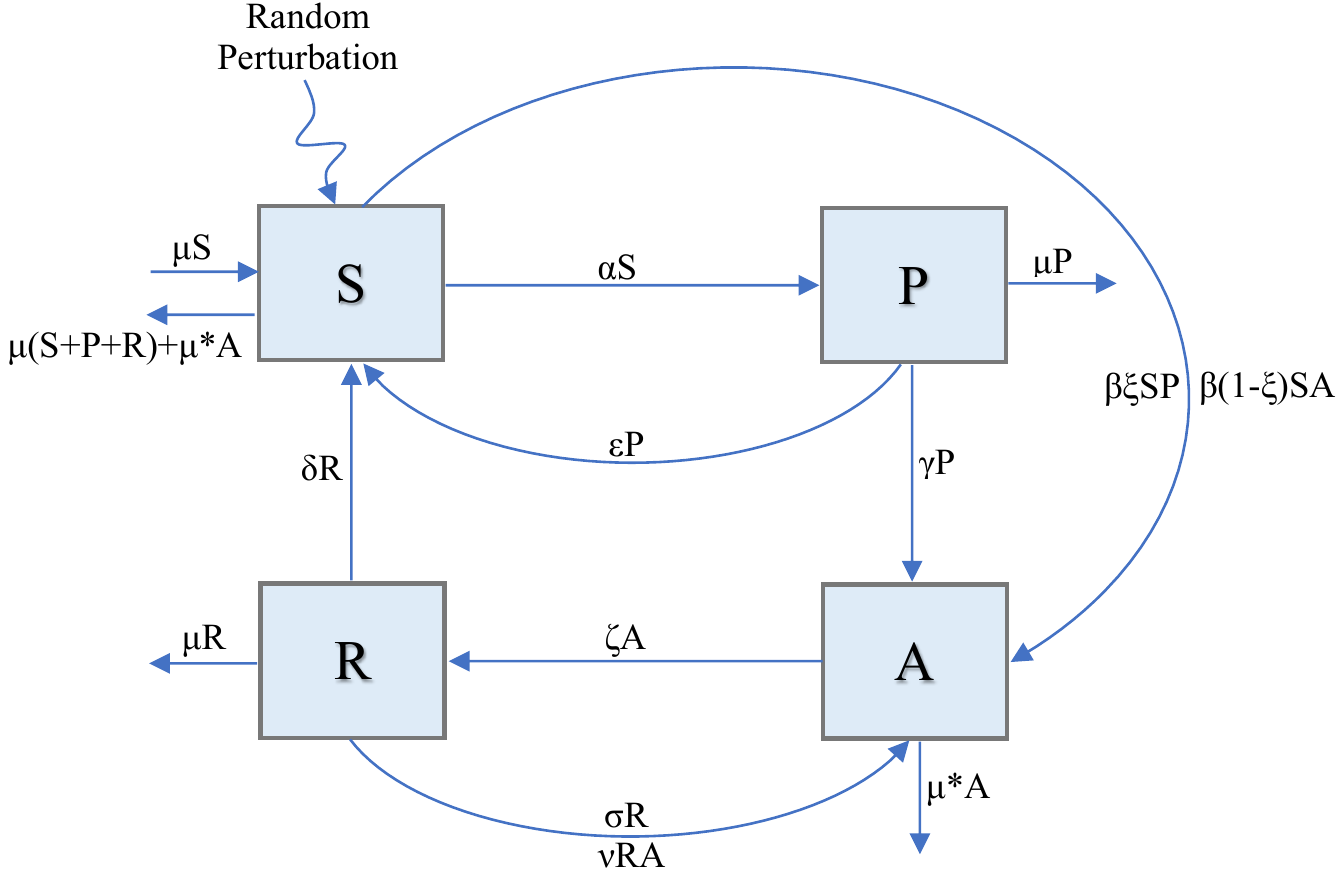}
\caption {\scriptsize A block diagram showing the relationships between the different groups in the compartmental model of opioid addiction with random perturbation (cf. Battista \etal \,\cite{r6})} \label{Fig1}
\end{figure}

Note that the normalized overall population is assumed constant (which is set to unity). Then, with $P(t) = 1 - S(t) - A(t) - R(t)$, we can reduce the above system of equations in Equations~\eqref{Eq2.1a}--\eqref{Eq2.1d} as follows
\begin{eqnarray}
\left.\begin{array}{l}
\dot{S}(t) = -\alpha S(t) -\beta(1-\xi)S(t) A(t) - \beta \xi S(t) (1 - S(t) - A(t) - R(t))  \\
\hspace{0.9in} + (\epsilon + \mu)(1 - S(t) - A(t) - R(t)) + (\delta + \mu) R(t) + \mu^{\ast} A(t)\\
\dot{A}(t) =  \gamma (1 - S(t) - A(t) - R(t)) + \sigma R(t) + \beta(1 - \xi) S(t) A(t) \\
 \hspace{0.9in} + \beta \xi S(t) (1 - S(t) - A(t) - R(t))  + \nu R A - (\zeta + \mu^{\ast})A(t) \\
\dot{R}(t) = \zeta A(t) - \mu R(t) A(t) - (\delta + \sigma + \mu) R(t)
\end{array}\right\}. \label{Eq2.2}
\end{eqnarray}
In order to facilitate our presentation, we adopt the following change of variables: $S \rightarrow x_1$, $A \rightarrow x_2$ and $R \rightarrow x_3$. Then, the system of equations in Equation~\eqref{Eq2.2} can be further rewritten as follows
\begin{eqnarray}
\left.\begin{array}{l}
\dot{x}_1(t) = f_1(x_1(t), x_2(t), x_3(t))\\
\dot{x}_2(t) = f_2(x_1(t), x_2(t), x_3(t)) \\
\dot{x}_3(t) = f_3(x_2(t), x_3(t))
\end{array}\right\} \label{Eq2.2b}
\end{eqnarray}
where the functions $f_1$, $f_2$ and $f_3$ are given by
\begin{align*}
& f_1(x_1(t), x_2(t), x_3(t)) \\
 &\quad \quad  =  -\alpha x_1(t) -\beta(1-\xi)x_1(t) x_2(t) - \beta \xi x_1(t) (1 - x_1(t) - x_2(t) - x_3(t)) \\
                                &\quad \quad\quad  + (\epsilon + \mu)(1 - x_1(t) - x_2(t) - x_3(t)) + (\delta + \mu) x_3(t) + \mu^{\ast} x_2(t),
\end{align*}
\begin{align*}
f_2(x_1(t), x_2(t), x_3(t)) =& \gamma (1 - x_1(t) - x_2(t) - x_3(t)) + \sigma x_3(t) + \beta(1 - \xi) x_1(t) x_2(t)  \\
                                 & \quad + \beta \xi x_1(t) (1 - x_1(t) - x_2(t) - x_3(t)) + \nu x_3(t) x_2(t) \\
                                 & \quad \quad \quad - (\zeta + \mu^{\ast}) x_2(t)
\end{align*}
and
\begin{align*}
f_3(x_2(t), x_3(t)) = \zeta x_2(t) - \mu x_3(t) x_2(t) - (\delta + \sigma + \mu) x_3(t),
\end{align*}
respectively.

In what follows, we assume that a random noise enters only through the dynamics of the susceptible group in Equation~\eqref{Eq2.2} and is then subsequently propagated to the other groups in the compartmental model (see also in Fig~\ref{Fig1}). To this end, we consider the corresponding system of stochastic differential equations (SDEs), i.e.,
\begin{eqnarray}
\left.\begin{array}{l}
dX_1(t) = f_1(X_1(t), X_2(t), X_3(t)) dt + \hat{\sigma}(X_1(t), X_2(t), X_3(t)) dW(t)\\
dX_2(t) = f_2(X_1(t), X_2(t), X_3(t)) dt \\
dX_3(t) = f_3(X_2(t), X_3(t)) dt
\end{array}\right\} \label{Eq2.6}
\end{eqnarray}
where $\bigl(W(t)\bigr)_{t \ge 0}$ is a one-dimensional Brownian motion, $\bigl(X_1(t), X_2(t), X_3(t)\bigr)_{t \ge 0}$ being an $\mathbb{R}^3$-valued degenerate diffusion process, and $\hat{\sigma}$ and $\hat{\sigma}^{-1}$ are assumed to be bounded functions. Moreover, if we denote by a bold letter a quantity in $\mathbb{R}^3$, for example, the solution in Equation~\eqref{Eq2.6} is denoted by $\bigl(\mathbf{X}(t)\bigr)_{t \ge 0} = \bigl(X_1(t), X_2(t), X_3(t)\bigr)_{t\ge 0}$, then we can rewrite Equation~\eqref{Eq2.6} as follows
\begin{align}
d \mathbf{X}(t) = \mathbf{F} (\mathbf{X}(t)) dt + B \hat{\sigma}(\mathbf{X}(t)) dW(t), \label{Eq2.7}
\end{align}
where $\mathbf{F} = \bigl[f_1, f_2, f_3\bigr]^T$ is an $\mathbb{R}^3$-valued function and $B$ stands for a column vector that embeds $\mathbb{R}$ into $\mathbb{R}^3$, i.e., $B = [1, 0, 0]^T$. Note that the corresponding degenerate elliptic operator for the diffusion process $\mathbf{X}(t)$ is given by
\begin{align}
\mathcal{L} (\cdot) (\mathbf{x}) = \frac{1}{2} \operatorname{tr}\Bigl \{a(\mathbf{x}) D_{x_1}^2 (\cdot) \Bigr\} + \sum\nolimits_{i=1}^3 f_{i}(\mathbf{x}) D_{x_i} (\cdot),  \label{Eq2.8}
\end{align}
where $a(\mathbf{x})=\hat{\sigma}(\mathbf{x})\,\hat{\sigma}^T(\mathbf{x})$, $D_{x_i}$ and $D_{x_1}^2$ (with $D_{x_1}^2 = \bigl({\partial^2 }/{\partial x_1 \partial x_1} \bigr)$) are the gradient and the Hessian (w.r.t. the variable $x_i$, for $i \in \{1,2,3\}$), respectively.

Let $D \subset \mathbb{R}^3$ be a given bounded open domain, with smooth boundary $\partial D$ (i.e., $\partial D$ is a manifold of class $C^2$), and let us denote by $C^{\infty}(D)$ the spaces of infinitely differentiable functions on $D$.

The following statements are standing assumptions that hold throughout the paper.
\begin{assumption} \label{AS1} ~\\\vspace{-3mm}
\begin{enumerate} [(a)]
\item The functions $\hat{\sigma}(\mathbf{x})$ and $\hat{\sigma}^{-1}(\mathbf{x})$ are bounded $C^{\infty}(\mathbb{R}^3\bigr)$-functions, with bounded first derivatives. Moreover, the least eigenvalue of $a(\mathbf{x})$ is uniformly bounded away from zero, i.e.,
\begin{align*}
  \mathbf{y}^Ta(\mathbf{x}) \mathbf{y} \ge \lambda \bigl\vert \mathbf{y} \bigr\vert^2, \quad \forall \mathbf{x}, \mathbf{y} \in \mathbb{R}^{3}, \quad \forall t \ge 0,
\end{align*}
for some $\lambda > 0$.
\item The operator in Equation~\eqref{Eq2.8} is hypoelliptic in $C^{\infty}(D)$ (e.g., see \cite{Hor67} or \cite{Ell73}).
\end{enumerate}
\end{assumption}

Note that the hypoellipticity assumption is in general related to a strong accessibility property of controllable nonlinear systems that are driven by white noises (e.g., see \cite{SusJu72} concerning the controllability of nonlinear systems, which is closely related to \cite{StrVa72} and \cite{IchKu74}; see also \cite[Section~3]{Ell73} and \cite[Theorem~2]{Ama79}). That is, the hypoellipticity assumption further implies that the diffusion process $\mathbf{X}(t)$ has a transition probability density with a strong Feller property.

\subsection{Minimum exit rates and principal eigenvalues} \label{S2(2)}
 In this subsection, we consider the following controlled version of SDE in Equation~\eqref{Eq2.7}, with the corresponding controlled-diffusion process $\bigl(\mathbf{X}_{0,\mathbf{x}}^{u}(t) \bigr)_{t \ge 0}$, i.e., 
\begin{align}
d \mathbf{X}_{0,\mathbf{x}}^{u}(t) = \bigl[\mathbf{F} (\mathbf{X}_{0,\mathbf{x}}^{u}(t)) + B u(t) \bigr]dt + B \hat{\sigma}(\mathbf{X}_{0,\mathbf{x}}^{u}(t)) dW(t), \,\, \mathbf{X}_{0,\mathbf{x}}^{u}(0) = \mathbf{x}, \label{Eq2.15}
\end{align}
where $u(\cdot)$ is a measurable control process from a set $\mathcal{U}$ which is $\mathbb{R}$-valued progressively measurable processes (i.e., a family of nonanticipative processes, for all $t > s$, $(W(t)-W(s))$ is independent of $u(r)$ for $r \le s$) and such that
\begin{align*}
\mathbb{E} \int_{0}^{\infty} \vert u(t)\vert^2 dt < \infty.
\end{align*}
Here, our main objective is to minimize the asymptotic exit rate with which the controlled-diffusion process $\mathbf{X}_{0,\mathbf{x}}^{u}(t)$ exits from the given bounded open domain $D$.

In what follows, we specifically consider a stationary Markov control $u(t) = v\bigl(\mathbf{X}_{0,\mathbf{x}}^{v}(t)\bigr) \in \mathcal{U}$, for $t \ge 0$, with some measurable map $v \colon \mathbb{R}^{3} \rightarrow \mathcal{U}$. Then, we suppose that the controlled-SDE in Equation~\eqref{Eq2.15} is composed with an admissible Markov control $v$. Furthermore, let $\tau_{D}$ be the first exit-time for the controlled-diffusion process $\mathbf{X}_{0,\mathbf{x}}^{v}(t)$ from the given bounded domain $D$, i.e., 
\begin{align}
\tau_{D} = \inf \Bigl\{ t > 0 \, \bigl\vert \, \mathbf{X}_{0,\mathbf{x}}^{v}(t) \in \partial D \Bigr\}. \label{Eq-2}
\end{align}
Notice that the extended generator for the controlled-diffusion process $\mathbf{X}_{0,\mathbf{x}}^{v}(t)$ is given by
\begin{eqnarray}
\mathcal{L}_{v} \bigl(\cdot\bigr) \bigl(\mathbf{x}\bigr) = \frac{1}{2} \operatorname{tr}\Bigl \{a(\mathbf{x}) D_{x_1}^2 (\cdot) \Bigr\} + \Bigl\langle \mathbf{F}(\mathbf{x}) + Bv(\mathbf{x}), D_{\mathbf{x}}(\cdot) \Bigr\rangle, \label{Eq-4}
\end{eqnarray}
where $D_{\mathbf{x}}(\cdot)$ denotes the gradient operator with respect to $\mathbf{x}$ (i.e., $D_{\mathbf{x}}(\cdot) \equiv [D_{x_1}(\cdot),\,D_{x_2}(\cdot),\,D_{x_3}(\cdot)]^T$).

Next, let us consider the following eigenvalue problem
\begin{align}
\left.\begin{array}{c}
  - \mathcal{L}_{v} \psi_{v} \bigl(\mathbf{x}\bigr) = \lambda_{v} \psi_{v}\bigl(\mathbf{x}\bigr) \quad \text{in} \quad D \quad \vspace{2mm} \\
  \quad \psi_{v}\bigl(\mathbf{x}\bigr) = 0 \quad \text{on} \quad \partial D  \label{Eq-7}
  \end{array}\right\} 
\end{align}
where the extended generator $\mathcal{L}_{v}$ is given in the above Equation~\eqref{Eq-4}. 

In the following section, i.e., Section~\ref{S3}, using Theorems~1.1, 1.2 and 1.4 from \cite{QuaSi08} (see also \cite[Proposition~3.2]{BefA15a}), we provide a condition for the existence of a unique principal eigenvalue $\lambda_{v} > 0$ and an eigenfunction $\psi_{v} \in W_{loc}^{2,p} \bigl(D\bigr) \cap C\bigl(\bar{D}\bigr)$ pairs for the eigenvalue problem in Equation~\eqref{Eq-7}, with zero boundary condition on $\partial D$. Notice that such an eigenvalue $\lambda_{v}$ is also related to the minimum asymptotic exit rate with which the controlled-diffusion process $\mathbf{X}_{0,\mathbf{x}}^{v}(t)$ exits from the bounded domain $D$, when the controlled-SDE in Equation~\eqref{Eq2.15} is composed with an admissible Markov control $v$.

\section{Main results} \label{S3}
In this section, we present our main results (i.e., Propositions~\ref{P-1} and \ref{P-2}) that characterize admissible solutions to the optimal control problem in Equation~\eqref{Eq-7}.

The following proposition establishes a connection between the minimum exit rate and with that of the principal eigenvalue for the extended generator $\mathcal{L}_{v}$ in Equation~\eqref{Eq-4}.

\begin{proposition} \label{P-1}
Suppose that an admissible Markov control $v$ is given. Then, the principal eigenvalue $\lambda_{v} $ for the extended generator $\mathcal{L}_{v} $, with zero boundary condition on $\partial D$, is given by 
\begin{align}
\lambda_{v} = - \limsup_{t \rightarrow \infty} \frac{1} {t} \log \mathbb{P}_{\mathbf{x}} \bigl\{\tau_{D} > t \bigr\},  \label{Eq-8}
\end{align}
where $\tau_{D}$ is the first exit-time for the controlled-diffusion process $\mathbf{X}_{0,\mathbf{x}}^{v}(t)$ from the given bounded domain $D$, i.e., $\tau_{D} = \inf \bigl\{ t > 0 \, \bigl\vert \, \mathbf{X}_{0,\mathbf{x}}^{v}(t) \in \partial D \bigr\}$; while the probability $\mathbb{P}_{\mathbf{x}}\bigl\{\cdot\bigr\}$ in Equation~\eqref{Eq-8} is conditioned on the initial condition $\mathbf{x} \in D$ as well as on the admissible Markov control $v\bigl(\mathbf{X}_{0,\mathbf{x}}^{v}(t)\bigr)$, for $t \in [0,\, \tau_{D})$.
\end{proposition}

\begin{proof}
For $\delta > 0$, let $D^{\delta} \subset D$ (with $D^{\delta} \cup \partial D^{\delta} \subset D$) be a bounded domain with smooth boundary, increasing to $D$ as $\delta \rightarrow 0$. Let 
\begin{align*}
\tau_{D^{\delta}} = \inf \Bigl\{ t > 0 \, \bigl\vert \, \mathbf{X}_{0,\mathbf{x}}^{v}(t) \in \partial D^{\delta} \Bigr\}.
\end{align*}
Then, applying Krylov's extension of the It\^{o}'s formula valid for any continuous functions from $W_{loc}^{2,p} \bigl(D\bigr)$, with $p \ge 2$ (e.g., see \cite[Chapter~2]{Bor89}; cf. \cite[Section~10, pp.~121--128]{Kry80}) and the optional sampling theorem
\begin{align*}
 \psi_{v} \bigl(\hat{\mathbf{x}}\bigr) &= \mathbb{E}_{\mathbf{x}} \Bigl\{\exp\bigl( \lambda_{v} (t \wedge \tau_{D^\delta}) \bigr) \psi_{v} \bigl(\mathbf{X}_{0,\mathbf{x}}^{v}(t \wedge \tau_{D^\delta})\bigr) \Bigr\},
\end{align*}
for some $\hat{\mathbf{x}} \in D$.

Letting $\delta \rightarrow 0$, then we have $\tau_{D^\delta} \rightarrow \tau_{D}$, {\it almost surely}, and
\begin{align*}
 \psi_{v} \bigl(\hat{\mathbf{x}}\bigr) 
 &= \mathbb{E}_{\mathbf{x}} \Biggl\{\exp\bigl( \lambda_{v} (t \wedge \tau_{D^\delta}) \bigr) \psi_{v} \Bigl(\mathbf{X}_{0,\mathbf{x}}^{v}(t \wedge \tau_{D^\delta})\Bigr) \Biggr\} \\
&= \mathbb{E}_{\mathbf{x}} \Biggl\{\exp\bigl( \lambda_{v} t \bigr) \,\psi_{v} \Bigl(\mathbf{X}_{0,\mathbf{x}}^{v}(t)\Bigr) \mathbf{1}\Bigl\{\tau_{D^\delta} > t \Bigr\} \Biggr\} \\
  & \le \Bigl\Vert \psi_{v} \Bigl(\mathbf{X}_{0,\mathbf{x}}^{v}(t)\Bigr) \Bigr\Vert_{\infty}  \exp\bigl( \lambda_{v} \, t\bigr)  \mathbb{P}_{\mathbf{x}}\Bigl\{\tau_{D^\delta} > t\Bigr\}.
\end{align*}
If we take the logarithm and divide both sides by $t$, then, further let $t \rightarrow \infty$, we have
\begin{align}
\lambda_{v} & \ge - \liminf_{t \rightarrow \infty} \frac{1} {t} \log \mathbb{P}_{\mathbf{x}} \Bigl\{\tau_{D} > t \Bigr\} \ge - \limsup_{t \rightarrow \infty} \frac{1} {t} \log \mathbb{P}_{\mathbf{x}} \Bigl\{\tau_{D} > t \Bigr\},  \label{Eq-9}
\end{align}
with $\delta \rightarrow 0$ (since $\tau_{D}^\delta \rightarrow \tau_{D}$, when $\delta \rightarrow 0$).

On the other hand, let $B_k \supset \bar{D} \equiv D \cup \partial D$ be an open domain with smooth boundary and let $\tau_{B_k}$ be the first exit-time for the controlled-diffusion process $\mathbf{X}_{0,\mathbf{x}}^{v}(t)$ from the domain $B_k$. Furthermore, let $ \psi_{v,B_k}$ and $\lambda_{v,B_k}$ be the principal eigenfunction-eigenvalue pairs for the eigenvalue problem of $\mathcal{L}_{v}$ on $B_k$, with $\psi_{v,B_k} \bigl(\hat{\mathbf{x}}\bigr) = 1$, for some $\hat{\mathbf{x}} \in D$.

Then, we have the following
\begin{align*}
& \psi_{v,B_k} \Bigl(\hat{\mathbf{x}}\Bigr) = \mathbb{E}_{\mathbf{x}} \Bigl\{\exp\bigl(\lambda_{v,B_k} t \bigr) \psi_{v,B_k} \Bigl(\mathbf{X}_{0,\mathbf{x}}^{v}(t)\Bigr) \mathbf{1}\Bigl\{\tau_{B_k} > t\Bigr\} \Bigr\} \\
  &\quad\quad\quad\quad \ge  \inf_{\mathbf{y} \in D}\Bigl\vert \psi_{v,B_k}\bigl(\mathbf{y} \bigr) \Bigr\vert  \exp\bigl(\lambda_{v,B_k}\, t\bigr)  \mathbb{P}_{\mathbf{x}} \Bigl\{\tau_{B_k} > t\Bigr\}.
\end{align*}
Thus, from Equation~\eqref{Eq-9}, we have the following 
\begin{align*}
\lambda_{v,B_k} & \le - \limsup_{t \rightarrow \infty} \frac{1} {t} \log \mathbb{P}_{\mathbf{x}} \Bigl\{\tau_{D} > t \Bigr\} \le - \liminf_{t \rightarrow \infty} \frac{1} {t} \log \mathbb{P}_{\mathbf{x}} \Bigl\{\tau_{D} > t \Bigr\}.
\end{align*}
Then, using Proposition~4.10 of \cite{QuaSi08}, we have $\lambda_{v, B_k} \rightarrow \lambda_{v} $ and $\tau_{B_k} \rightarrow \tau_{D}$ as $B_k \rightarrow \bar{D}$. This completes the proof of Proposition~\ref{P-1}.
\end{proof}

\begin{remark} \label{R1}
In order to confine the controlled-diffusion process $\mathbf{X}_{0,\mathbf{x}}^{v}(t)$ for a longer duration in a given bounded domain $D$, a standard approach is to maximize the mean exit-time, i.e., $\max_{u \in \mathcal{U}}\,\mathbb{E}_{\mathbf{x}} \bigl\{\tau_{D} \bigr\}$, from the bounded domain $D$. Note that, in general, it is difficult to get effective information about a minimum exit probability and, at the same time, a set of admissible Markov controls in this way. On the other hand, we also observe that a more suitable objective would be to minimize the asymptotic exit rate with which the controlled-diffusion process $\mathbf{X}_{0,\mathbf{x}}^{v}(t)$  exits from the bounded domain $D$, where we argued that the corresponding optimal control problem can be closely identified with a nonlinear eigenvalue problem.
\end{remark}

In what follows, let us define the following HJB equation
\begin{align}
&\mathcal{L}_{u} \bigl(\cdot\bigr) \bigl(\mathbf{x}, u\bigr) = \frac{1}{2} \operatorname{tr}\Bigl \{a(\mathbf{x}) D_{x_1}^2 (\cdot) \Bigr\} + \Bigl\langle \mathbf{F}(\mathbf{x}) + Bu, D_{\mathbf{x}}(\cdot) \Bigr\rangle, \label{Eq-10}
\end{align}
with $D_{\mathbf{x}}(\cdot) = [D_{x_1}(\cdot),\,D_{x_2}(\cdot),\,D_{x_3}(\cdot)]^T$.

Note that we can also associate the above HJB equation with the following optimal control problem
\begin{align}
\max_{u \in \mathbb{R}} \Bigl \{ \mathcal{L}_{u}\psi\bigl(\mathbf{x}, u\bigr) + \lambda \psi\bigl(\mathbf{x}\bigr) \Bigr\}. \label{Eq-11}
\end{align}

Then, we have the following result that provides a sufficient condition for admissible optimal Markov control.
\begin{proposition} \label{P-2}
There exist a unique $\lambda^{\ast} > 0$ (which is the minimum exit rate) and $\psi^{\ast} \in C^2(D) \cap C(\bar{D})$, with $\psi^{\ast} > 0$ on $D$, that satisfies the optimal control problem in Equation~\eqref{Eq-11}. Moreover, the admissible Markov control $v^{\ast}$ is optimal if and only if $v^{\ast}$ is a measurable selector for 
\begin{align}
\argmax \Bigl\{\mathcal{L}_{u} \psi^{\ast} \bigl(\mathbf{x}, \,\cdot \,\bigr) \Bigr\}, \,\, \mathbf{x} \in D. \label{Eq-12}
\end{align}
\end{proposition}

\begin{proof}
The first claim for $\psi^{\ast} \in W_{loc}^{2,p} \bigl(D\bigr) \cap C\bigl(\bar{D}\bigr)$, with $p>2$, follows from Equation~\eqref{Eq-7} (cf. \cite[Theorems~1.1, 1.2 and 1.4]{QuaSi08}). Notice that if $v^{\ast}$ is measurable selector of $\argmax \bigl\{\mathcal{L}_{u} \psi^{\ast} \bigl(\mathbf{x}, \,\cdot \,\bigr) \bigr\}$, with $\mathbf{x} \in D$ (e.g., see also \cite{Ben70} on the measurable selection theorem based on specified information about the state trajectories). Then, by the uniqueness claim for eigenvalue problem in Equation~\eqref{Eq-7}, we have
\begin{align*}
\lambda_{v^{\ast}} = - \limsup_{t \rightarrow \infty} \frac{1} {t} \log \mathbb{P}_{\mathbf{x}} \Bigl\{\tau_{D} > t \Bigr\}, 
\end{align*}
where the the probability $\mathbb{P}_{\mathbf{x}} \bigl\{\cdot\}$ is conditioned with respect to $\mathbf{x}$ and $v^{\ast}$. Then, for any other admissible control $u$, we have
\begin{align*}
\mathcal{L}_{u} \psi^{\ast} \bigl(\mathbf{x}, u \bigr) + \lambda_{v^{\ast}} \psi^{\ast}\bigl(\mathbf{x}\bigr) \le 0, \quad \forall t \ge 0.
\end{align*}
Let $Q \subset \mathbb{R}^3$ be a smooth bounded open domain containing $\bar{D}$. Let $\hat{\psi}$ and  $\hat{\lambda}$ be the principal eigenfunction-eigenvalue pairs for the eigenvalue problem of $\mathcal{L}_{u}$ on $\partial Q$. 

Let 
\begin{align*}
\tau_{Q} = \inf \Bigl\{ t > 0 \, \bigl\vert \, \mathbf{X}_{0,\mathbf{x}}^{u}(t) \in \partial Q \Bigr\}.
\end{align*}
Then, under $u$, we have  
\begin{align*}
& \hat{\psi}\bigl(\mathbf{x}\bigr)\ge \mathbb{E}_{\mathbf{x}} \Biggl\{\exp\bigl( \hat{\lambda} t \bigr) \,\hat{\psi} \bigl(\mathbf{X}_{0,\mathbf{x}}^{u}(t)\bigr) \mathbf{1}\Bigl\{\tau_{Q} > t\Bigr\} \Biggr\} \\
  &\quad\quad \ge  \inf_{\mathbf{y} \in D}\Bigl\vert \hat{\psi}\bigl(\mathbf{y}\bigr) \Bigr\vert  \exp\bigl(\hat{\lambda}\, t\bigr) \mathbb{P}_{\mathbf{x}}\Bigl\{\tau_{Q} > t\Bigr\}.
\end{align*}
Leading to
\begin{align*}
\hat{\lambda} \le - \limsup_{t \rightarrow \infty} \frac{1} {t} \log \mathbb{P}_{\mathbf{x}} \Bigl\{\tau_{D} > t \Bigr\}.
\end{align*}
Letting $Q$ shrink to $D$ and using Proposition~4.10 of \cite{QuaSi08}, then we have $\hat{\lambda} \rightarrow \lambda_{v^{\ast}}$.Then, we have
\begin{align*}
\lambda_{v^{\ast}} = - \limsup_{t \rightarrow \infty} \frac{1} {t} \log \mathbb{P}_{\mathbf{x}} \Bigl\{\tau_{D} > t \Bigr\},
\end{align*}
which establishes the optimality of $v^{\ast}$ and the fact that $\lambda_{v^{\ast}}$ is the minimum exit rate.

Conversely, let $\hat{v}^{\ast}$ be any optimal Markov control. Then, we have
\begin{align*}
\mathcal{L}_{\hat{v}^{\ast}} \hat{\psi} \bigl(\mathbf{x}, \hat{v}^{\ast}(\mathbf{x}) \bigr) + \hat{\lambda}_{\hat{v}^{\ast}} \, \hat{\psi}\bigl(\mathbf{x} \bigr) = 0
\end{align*}
and
\begin{align*}
\mathcal{L}_{v^{\ast}} \psi^{\ast} \bigl(\mathbf{x}, \hat{v}^{\ast}(\mathbf{x}) \bigr) + \lambda_{v^{\ast}} \, \psi^{\ast} \bigl(\mathbf{x} \bigr) \le 0, \quad \forall t > 0,
\end{align*}
with $\hat{\lambda}_{\hat{v}^{\ast}} = \lambda_{v^{\ast}}$.

Furthermore, notice that $\psi^{\ast}$ is a scalar multiple of $\hat{\psi}$ and, at $\mathbf{x} \in D$ (cf. \cite[Theorem~1.4(a)]{QuaSi08}). Then, we see that $\hat{v}^{\ast}$ is also a maximizing measurable selector in Equation~\eqref{Eq-11}. This completes the proof of Proposition~\ref{P-2}.
\end{proof}

\begin{remark} \label{R2}
Note that the above proposition (which is a verification theorem) is useful for selecting the most appropriate admissible Markov control that confines the controlled-diffusion process $\mathbf{X}_{0,\mathbf{x}}^{v}(t)$ to the prescribed bounded domain $D$ for a longer duration.
\end{remark}

\section{Further remarks} \label{S4}
This paper briefly considered the problem of controlling a diffusion process pertaining to an opioid epidemic dynamical model with random perturbation so as to prevent it from leaving a given bounded open domain. Here, we specifically argued that the problem can be posed as minimizing the asymptotic exit rate (as opposed to maximizing the mean exit-time) with which the controlled-diffusion process exits from the given bounded domain and we further established a connection with a nonlinear eigenvalue problem. Moreover, we also proved a verification theorem that provides a sufficient condition for the solution of the corresponding HJB equation and minimizing admissible controls.

Here, it is worth remarking that there are many possible directions for extensions, for example, one obvious extension would be to consider a risk-sensitive version of the mean escape time criterion in the sense of Dupuis and McEneaney \cite{DupM97}, when the randomly perturbed opioid epidemic dynamics obeys a controlled-SDE with coefficients depending on a small parameter $\epsilon \ll 1$ (cf. Equation~\eqref{Eq2.15}), i.e., 
\begin{align*}
d \mathbf{X}_{0,\mathbf{x}}^{\epsilon,u}(t) = \bigl[\mathbf{F} (\mathbf{X}_{0,\mathbf{x}}^{\epsilon,u}(t)) + B u(t) \bigr]dt + \sqrt{\epsilon} B \hat{\sigma}(\mathbf{X}_{0,\mathbf{x}}^{\epsilon,u}(t)) dW(t), \,\, \mathbf{X}_{0,\mathbf{x}}^{\epsilon,u}(0) = \mathbf{x}.
\end{align*}
Note that the natural optimization criterion is to minimize the exponential function of the escape time to a critical loss threshold, i.e.,
\begin{align*}
\mathbf{E}_{\mathbf{x}}^{\epsilon} \exp \Bigl \{ -\frac{\theta \tau_{D}^{\epsilon}}{\epsilon} \Bigr \},
\end{align*}
where $\theta$ is a positive design parameter, $\tau_{D}^{\epsilon} = \inf \bigl\{ t > 0 \, \bigl\vert \, \mathbf{X}_{0,\mathbf{x}}^{\epsilon, u}(t) \in \partial D \bigr\}$, and $\mathbf{E}_{\mathbf{x}}^{\epsilon} \bigl \{\cdot\bigr\}$ denotes the expectation conditioned on $\mathbf{X}_{0,\mathbf{x}}^{\epsilon,u}(t)$. Equivalently, we can also consider maximizing the following criterion
\begin{align*}
-\epsilon \log \mathbf{E}_{\mathbf{x}}^{\epsilon} \exp \Bigl \{ -\frac{\theta \tau_{D}^{\epsilon}}{\epsilon} \Bigr \},
\end{align*}
where the risk-sensitive problem is to obtain an admissible optimal control for the following value
\begin{align*}
\max_{u \in \mathcal{U}} -\epsilon \log \mathbf{E}_{\mathbf{x}}^{\epsilon} \exp \Bigl \{ -\frac{\theta \tau_{D}^{\epsilon}}{\epsilon} \Bigr \} \quad \text{as} \quad \epsilon \rightarrow 0.
\end{align*}
which is also averse to any rapid escapes from the given bounded domain $D$. Moreover, along this direction, one could also exploit the connection between viscosity solutions and that of the theory of large deviations, where the escape time control can be posed as a stochastic differential game (e.g., see Bou\'{e} and Dupuis \cite{BouD01} for additional discussions related to time-consistency of such admissible optimal controls). 

Finally, we emphasize that obtaining qualitative information on the asymptotic exit rate (including the first-exit time from $D$ and the first-exit location on $\partial D$) for the diffusion process pertaining to an opioid epidemic dynamics with random perturbation could be useful for developing evidence-based strategies that aim at curbing opioid epidemics or assisting in interpreting outcome results from opioid-related policies. Moreover, results -- based on such qualitative information -- are more practical for characterizing typical sample paths of regular prescription opioid users, opioid addicts or the process of rehabilitation and relapsing into opioid drug uses.

\end{document}